\documentclass[oneside,a4paper,12pt,reqno]{amsart}
\usepackage{amsfonts}
\usepackage{amssymb}
\usepackage{amsthm}
\usepackage{amsmath}
\usepackage{a4wide}
\usepackage{mathrsfs}
\usepackage{epsfig}
\usepackage{palatino}
\usepackage{tikz,fp,ifthen}
\usepackage{float}
\usepackage{graphicx}
\usepackage{enumitem}

\usepackage{color}
\definecolor{citegreen}{rgb}{0,0.6,0}
\definecolor{refred}{rgb}{0.8,0,0}
\usepackage[colorlinks, urlcolor=blue, citecolor=citegreen, linkcolor=refred]{hyperref}

\usepackage{mathtools}
\mathtoolsset{showonlyrefs}
\newcommand{\AAA}{{\mathrm{A}}} % Seconda forma
\newcommand{\EE}{{\mathrm{E}}} % Gradiente del funzionale
\newcommand{\HHH}{{\mathrm{H}}} % Curv media
\newcommand{\RR}{{\mathrm{R}}} % Riemann tensor
\def\RRR{{\mathrm R}}

\newcommand{\N}{\mathbb{N}}

\newcommand{\R}{\mathbb{R}}
\newcommand{\SSS}{\mathbb{S}}

\newcommand{\cE}{\mathcal{E}}
\newcommand{\cF}{\mathcal{F}}

\newcommand{\cL}{\mathcal{L}}

\newcommand{\cp}{\mathfrak{p}}
\newcommand{\cq}{\mathfrak{q}}

\newcommand{\id}{{\mathrm{id}}}

\newcommand{\ep}{\varepsilon}

\newcommand{\de}{\delta}

\newcommand{\vfi}{\varphi}

\newcommand{\ro}{\rho}

\newcommand{\De}{\Delta}

\newcommand{\Om}{\Omega}

\DeclarePairedDelimiter\scal{\langle}{\rangle}

\theoremstyle{plain}
\newtheorem{thm}{Theorem}[section] 

\newtheorem{prop}[thm]{Proposition}
\newtheorem{lemma}[thm]{Lemma}
\newtheorem{cor}[thm]{Corollary}

\theoremstyle{definition}

\theoremstyle{remark}

\newtheorem{remark}[thm]{Remark}

\numberwithin{equation}{section}

\title[Asymptotic convergence of evolving hypersurfaces]{Asymptotic convergence of evolving hypersurfaces}

\author{Carlo Mantegazza}
\address{Carlo Mantegazza\\
	Dipartimento di Matematica e Applicazioni,
	Universit\`a di Napoli Federico II, Via Cintia, Monte S. Angelo 80126 Napoli,
	Italy}
\email{c.mantegazza@sns.it}

\author{Marco Pozzetta}
\address{Marco Pozzetta\\
	Dipartimento di Matematica e Applicazioni,
	Universit\`a di Napoli Federico II, Via Cintia, Monte S. Angelo 80126 Napoli,
	Italy
}
\email{marco.pozzetta@unina.it}

\date{\today}

\begin{document}

\begin{abstract}
If $\psi:M^n\to \mathbb{R}^{n+1}$ is a smooth immersed closed hypersurface, we consider the functional
\[
\mathcal{F}_m(\psi) = \int_M 1 + |\nabla^m \nu |^2 \, d\mu,
\]
where $\nu$ is a local unit normal vector along $\psi$, $\nabla$ is the Levi--Civita connection of the Riemannian manifold $(M,g)$, with $g$ the pull--back metric induced by the immersion and $\mu$ the associated volume measure. We prove that if $m>\lfloor n/2 \rfloor$ then the unique globally defined smooth solution to the $L^2$--gradient flow of $\mathcal{F}_m$, for every initial hypersurface, smoothly converges asymptotically to a critical point of $\mathcal{F}_m$, up to diffeomorphisms. The proof is based on the application of a \L ojasiewicz--Simon gradient inequality for the functional $\mathcal{F}_m$.
\end{abstract}

\maketitle

\setcounter{tocdepth}{1}
\tableofcontents

\section{Introduction}

We consider a closed connected differentiable manifold $M^n$ of dimension $n\ge 1$ and $\psi:M^n\to \R^{n+1}$ a smooth immersion of $M^n$ in the Euclidean space $\R^{n+1}$. We shall usually omit the superscript $n$ denoting the dimension of $M$. For such an immersion, we always assume that $M$ is endowed with the metric tensor $g=\psi^*\scal{\cdot,\cdot}_{\R^{n+1}}$, induced by the immersion $\psi$ that we also sometimes simply denote as $\scal{\cdot,\cdot}$. The Levi--Civita connection of the Riemannian manifold $(M,g)$ is denoted by $\nabla$ and the associated volume measure by $\mu$. Then, for $m\in \N$ with $m\ge 1$, we consider the functional
\[
\cF_m(\psi)\coloneqq \int_M  1+ |\nabla^m\nu|^2 \,d\mu
\]
on the smooth immersions $\psi:M^n\to \R^{n+1}$, where $\nu$ is a (locally defined) unit normal vector field along $\psi$. Let us specify that in the above definition, if $\nu=\nu^\alpha e_\alpha$ and $e_\alpha$ is the standard basis in $\R^{n+1}$, we mean $|\nabla^m\nu|^2\coloneqq \sum_{\alpha=1}^{n+1} |\nabla^m\nu^\alpha|^2$. Notice that $\cF_m$ is independent of the local choice of the unit normal $\nu$ and it is well-defined without further hypotheses on $M$. However, by the discussion below in Remark \ref{rem:Orientability}, we shall always assume without loss of generality that $M$ is orientable and that a global choice of unit normal field $\nu$ is understood.

We observe that in case $n=m=1$, we recognize the well-known elastic energy of closed curves $\psi:\SSS^1\to \R^2$, i.e., $\mathcal{F}_1(\psi)=\int_{\SSS^1} 1+|k|^2 d s$, where $k$ is the curvature of $\psi$. If instead $n=2$, $m=1$, then one has $|\nabla \nu|^2=|\AAA|^2$, where $\AAA$ is the second fundamental form (see \eqref{gwein}), and then $\mathcal{F}_1$ yields the sum of the area and of (an equivalent form of) the Willmore energy of an immersed surface $\psi:M^2\to \R^3$. Let us also notice that if instead $n=m=2$, then $\mathcal{F}_2(\psi)=\int_{M} 1+|\nabla \AAA|^2 +\cp_0(\AAA,\AAA,\AAA,\AAA) d \mu$ for any given immersion $\psi:M^2\to \R^3$, where $\cp_0$ is some polynomial as in \eqref{eq:DefP} below.

By the formula for the first variation of $\cF_m$ (see Theorem~\ref{thm:FirstVariation}), one can prove that the associated $L^2$--gradient flow is defined by an evolution equation
\[
\frac{\partial\vfi}{\partial t}(p,t) = - \EE_m(\vfi_t)(p)\nu_t(p)
\]
for a smooth map $\vfi:M\times[0,T)\to \R^{n+1}$ (where $\varphi_t=\varphi(\cdot,t):M\to\R^{n+1}$ describes the moving hypersurface and $\nu_t$ is its unit normal vector field), which turns out to be a (quasilinear and degenerate) parabolic system of PDEs.\\
If $m> \lfloor n/2 \rfloor$, the study carried out in~\cite{Ma02} shows that for every initial smooth immersed hypersurface $\vfi_0:M\to\R^{n+1}$, there exists a unique smooth solution $\vfi_t$ with initial datum $\vfi_0$, defined for all positive times; moreover, $\vfi_t$ \emph{sub--converges} to a critical point $\vfi_\infty:M\to \R^{n+1}$ of the functional $\cF_m$, that is, such that $\EE_m(\vfi_\infty)=0$ (see Theorem~\ref{thm:Subconvergence}). By {\em sub--convergence} we mean that for some sequence of times $t_j\to+\infty$, the sequence $\vfi_{t_j}$ smoothly converges to $\vfi_\infty$, up to diffeomorphisms and translations in $\R^{n+1}$. More precisely, there exist a sequence of smooth diffeomorphisms $\sigma_j:M\to M$ and a sequence of points $p_j\in\R^{n+1}$ such that the sequence of immersions $\vfi_{t_j}\circ \sigma_j - p_j$ converge to $\vfi_\infty$ in $C^{k}(M)$, for any $k\in \N$. From such a sub--convergence result it is anyway not possible to immediately deduce that the flow {\em fully converges}, i.e., that there exists the full limit of $\vfi_t$ as $t\to+\infty$ in $C^k(M)$ for any $k$ (up to diffeomorphisms). Actually, the sub--convergence of the flow does not even guarantee that the limits of the flow along different diverging sequences of times coincide. Moreover, as the evolution equations involved here are of order greater or equal than four with respect to the parametrization, it is not even possible to conclude that the flow stays in a compact set of $\R^{n+1}$ for all times by means of comparison arguments, as maximum principles are not applicable.

In this work we address this issue, that is, we prove that the gradient flow of $\cF_m$ does actually converge, for any initial hypersurface. Our main result is the following theorem.

\begin{thm}\label{thm:Convergence}
Let $\vfi_0:M^n\to\R^{n+1}$ be a smooth immersion of a closed hypersurface and let $m > \lfloor n/2 \rfloor$. Then the unique smooth solution $\vfi:M\times [0,+\infty) \to \R^{n+1}$ to the evolution problem
\[
\begin{cases}
\frac{\partial\vfi}{\partial t} = - \EE_m(\vfi_t) \nu_t\\
\vfi(0,\cdot)=\vfi_0
\end{cases}
\]
converges in $C^k(M)$ to a smooth critical point $\vfi_\infty:M\to \R^{n+1}$ of $\cF_m$ as $t\to+\infty$, for every $k\in \N$ up to diffeomorphisms of $M$; more precisely, there exists a one-parameter family of diffeomorphisms $\sigma_t:M\to M$ such that the flow $\vfi_t\circ \sigma_t$ converges in $C^k(M)$ to a smooth critical point $\vfi_\infty$ of $\cF_m$ as $t\to+\infty$, for every $k\in \N$.\\
In particular, there exists a compact set $K\subseteq \R^{n+1}$ such that $M_t=\varphi_t(M)\subseteq K$ for any $t\ge0$.
\end{thm}

We remark that the assumption $m > \lfloor n/2 \rfloor$ is sharp, in fact if $m\le  \lfloor n/2 \rfloor$ then one gets flows that may develop singularities in finite time.

A relevant motivation for the study of the gradient flow of the functionals $\cF_m$ goes back to Ennio~De~Giorgi. In one of his last papers, he conjectured that any compact $n$--dimensional hypersurface in $\R^{n+1}$, evolving by the gradient flow of certain functionals depending on sufficiently high derivatives of the curvature does not develop singularities during the flow (see~\cite{DeGiorgiIT} and~\cite[Section~5, Conjecture~2]{DeGiorgiEN} for an English translation, see also~\cite[Section~9]{Ma02}). This result was central in his program to approximate singular geometric flows, as the mean curvature flow, with sequences of smooth ones (see~\cite[Sec. 9]{Ma02} and~\cite{BellMantegNovaga} for a result in this direction).
The functionals $\cF_m$ are strictly related to the ones proposed by De~Giorgi since, roughly speaking, the derivative of the normal field yields the curvature of $M$ (see~\eqref{gwein}). Though not exactly the same, the energies $\cF_m$ can then play the same role in the approximation process he suggested and the analysis of the asymptotic behavior of their gradient flow is another step in understanding such process.

The main tool in the proof of Theorem~\ref{thm:Convergence} is a \emph{\L ojasiewicz--Simon gradient inequality} for the functional $\cF_m$ (see Corollary~\ref{cor:Loja}). Such an estimate bounds a {\em less--than--$1/2$} power of the difference in ``energy'' (the value of the functional) between a critical point and a point sufficiently close to it in terms of a suitable norm of the first variation of the functional. The use of this kind of inequalities in the study of the convergence of parabolic equations of gradient--type goes back to \L ojasiewicz~\cite{Loja63, Loja84} and to the seminal paper of Simon~\cite{Si83}. More recently, useful sufficient hypotheses implying a \L ojasiewicz--Simon gradient inequality have been derived in~\cite{Ch03} (see also \cite{FeehanMaridakis} for several recent generalizations). Building on the abstract tools developed in~\cite{Ch03}, a first recent application of the inequality to get convergence of an extrinsic geometric flow is contained in~\cite{ChFaSc09}, where the authors investigate the Willmore flow of surfaces in neighborhoods of critical points. In the last years and in the context of higher order geometric gradient flows, the \L ojasiewicz--Simon inequality appeared as a tool for ``promoting'' the sub--convergence of a flow to its full convergence. As applications of this method we mention~\cite{DaPoSp16}, in which it is proved the full convergence of the elastic flow of open clamped curves, and~\cite{Po20Loja}, in which the sub--convergence of the $p$--elastic flow of closed curves on Riemannian manifolds is shown to imply the full convergence. The analysis in~\cite{Po20Loja} led to a further simplification and deeper understanding of the method, which is exposed in~\cite{MaPo20Loja}. In this work we essentially generalize the strategy employed in~\cite{MaPo20Loja} to the gradient flows of the functionals $\cF_m$. Moreover we tried to keep most of the arguments as general as possible, in order that the method could be possibly applied also to other geometric gradient flows, also in the context of extrinsic geometric flows in higher codimension possibly in Riemannian manifolds.

In the broad framework of geometric flows, \L ojasiewicz--Simon gradient inequalities found many other notable applications. For example, the study of singularities of mean curvature flow can be reconducted to the study of the smooth convergence of a suitable extrinsic geometric flow. Smooth convergence of such flow has been proved exploiting \L ojasiewicz--Simon inequalities in some relevant particular cases in \cite{Schulze, ColdingMinicozziUniqueness, ChodoshSchule}. Let us also mention \cite{ChoiMantoulidis}, where a classification of ancient solutions to a family of geometric flows in Riemannian manifolds is derived. \L ojasiewicz--Simon inequalities have been employed also in the context of intrinsic geometric flows. We refer, for instance, to the study of the rate of convergence of Yamabe flows in \cite{CarlottoChodoshRub}, or to the deep investigation on the Yang--Mills flow contained in \cite{FeehanGlobalYM} (see also references therein).

\subsection*{Notation and geometry of submanifolds} Let $M^n$ be closed, connected, and orientable. Let $\psi:M\to \R^{n+1}$ be a smooth immersion of $M$ and let $\nu$ be global unit normal field on $M$ along $\psi$.

\begin{remark}\label{rem:Orientability}
In case $M$ is not orientable, given an initial immersion $\vfi_0:M\to \R^{n+1}$, we can consider the canonical two--fold cover $\pi:\widetilde{M}\to M$, where $\widetilde{M}$ is orientable and the initial immersion $\widetilde{\vfi}_0= \vfi_0 \circ \pi$. By uniqueness of the flow $\vfi_t$ starting at $\vfi_0$ (Theorem~\ref{thm:Subconvergence}), it follows that the flow $\widetilde{\vfi}_t$ starting at $\widetilde{\vfi}_0$ is just $\widetilde{\vfi}_t = \vfi_t \circ \pi$. Therefore, if we prove that $\widetilde{\vfi}_t$ smoothly converges, then the same holds for the flow $\vfi_t$. Hence, also in this case Theorem~\ref{thm:Convergence} holds.	
\end{remark}

As the metric $g$ is obtained pulling it back with $\psi$, in local coordinates $\{x_i\}$ on $M$, we have 
$$
g_{ij}(x)=\left\langle\frac{\partial{{\psi}}(x)}{\partial
    x_i},\frac{\partial{{\psi}}(x)}{\partial
    x_j}\right\rangle
$$
and the canonical volume measure induced by the metric $g$ is given in local coordinates by $\mu=\sqrt{\det(g_{ij})\,}\,{\mathcal L}^n$ where 
${\mathcal L}^n$ is the standard Lebesgue measure on $\R^n$.\\
The induced covariant derivative on $(M,g)$ of a tangent vector field $X$ is given by
$$
\nabla_jX^i=\frac{\partial\,}{\partial
  x_j}X^i+\Gamma^{i}_{jk}X^k
$$
(in the whole paper we will adopt the Einstein convention of summation over repeated indices) 
where the Christoffel symbols $\Gamma^{i}_{jk}$ are expressed by the formula
$$
\Gamma^{i}_{jk}=\frac{1}{2} g^{il}\left(\frac{\partial\,}{\partial
    x_j}g_{kl}+\frac{\partial\,}{\partial
    x_k}g_{jl}-\frac{\partial\,}{\partial
    x_l}g_{jk}\right)\,.
$$
We will write $\partial_i$ for the coordinates derivatives, opposite to the covariant ones $\nabla_i$.
With $\nabla^k T$ we will mean the $k$--th iterated covariant derivative of a tensor $T$.
If $f$ is a smooth function on a smooth immersed hypersurface, the symbol $\nabla f$ denotes its gradient and $\nabla^2 f$ its Hessian, whose trace is the Laplacian $\Delta f$.

The second fundamental form $\AAA$ of the immersion $\psi$ is the bilinear symmetric form acting on any pair of tangent vector fields $X,Y$ to the hypersurface as
\[
\AAA(X,Y)=-\bigl\langle\nabla^{\R^{n+1}}_X Y,\nu\rangle,
\]
given a (global, since we assumed $M$ orientable) choice of the unit normal vector $\nu$ (we will usually identify $TM$ with $d\psi(TM)\subseteq\R^{n+1}$ and in this formula the field $Y$ is extended locally around $\psi(M)$ in $\R^{n+1}$). Hence $\AAA$ is defined up to a sign, that is, up to the choice of $\nu$, while $\AAA\nu$ is independent of the choice of $\nu$. In local coordinates, the components $ h_{ij}$ of $\AAA$ are given by
$$
 h_{ij}(x)=-\left\langle\nu(x),\frac{\partial^2{{\psi}}(x)}
{\partial x_i\partial x_j}\right\rangle\,.
$$
We recall that the following Gauss--Weingarten relations hold
\begin{equation}\label{gwein}
\partial^2_{ij}\psi=\Gamma^k_{ij} \partial_k \psi- h_{ij}\nu,
\qquad
\partial_i \nu = h_{ij}g^{jk}\partial_k \psi.
\end{equation}
The mean curvature $\HHH$ of $\psi$ is the trace of $\AAA$, that is
$$
\HHH(x)=g^{ij}(x) h_{ij}(x).
$$

By means of the Gauss equation, the Riemann tensor can be 
expressed via the second fundamental form, in local coordinates, as follows
\begin{align*}
\RRR_{ijkl}\,=&\,h_{ik}h_{jl}-h_{il}h_{jk}\,.
\end{align*}
Hence, the formulae for the interchange of covariant derivatives become
$$
\nabla_i\nabla_jX^s-\nabla_j\nabla_iX^s=\RRR_{ijkl}g^{ks}X^l=\RRR_{ijl}^sX^l=
\left(h_{ik}h_{jl}-h_{il}h_{jk}\right)g^{ks}X^l\,,
$$
\begin{equation}\label{ichange}
\nabla_i\nabla_j\omega_k-\nabla_j\nabla_i\omega_k=\RRR_{ijkl}g^{ls}\omega_s=\RRR_{ijk}^s\omega_s=
\left(h_{ik}h_{jl}-h_{il}h_{jk}\right)g^{ls}\omega_s\,,
\end{equation}
where we recall that by $\nabla_i\nabla_j X^s$ we mean the $s$-th component of the field $(\nabla^2X)(\partial_i,\partial_j)$.

Abusing a little the notation, if $T_1,...,T_N$ is a finite family of tensors, we denote by
\[
\circledast_{k=1}^N T_k \coloneqq T_1 * \ldots * T_N
\]
a generic contraction of some indices of the tensors $T_1,...,T_N$ using the coefficients $g_{ij}$ or $g^{ij}$. We will also denote
\begin{equation}\label{eq:DefP}
	\cp_s(T_1,\ldots, T_N) \coloneqq \sum_{i_1+\ldots+i_N=s} C_{i_1,\ldots, i_N} \nabla^{i_1} T_1 * \ldots * \nabla^{i_N} T_N,
\end{equation}
for some constants $C_{i_1,\ldots, i_N} \in \R$. Notice that in every additive term of $\cp_s(T_1,\ldots, T_N)$ each tensor appears exactly once (there are no repetitions).\\
We will use instead the symbol $\cq^s(T_1,\ldots,T_N)$ for ``polynomials'' of the form
\[
\cq^s(T_1,\ldots,T_N) \coloneqq \sum \left[ \circledast_{i_1=1}^{M_1} \nabla^{i_1} T_1 \, \ldots \, \circledast_{i_N=1}^{M_N} \nabla^{i_N} T_N  \right],
\]
with $M_j\ge 1$ for any $j=1,...,N$ and with
\[
s = \sum_{i_1=1}^{M_1} (i_1+1)
+ \ldots +
\sum_{i_N=1}^{M_N} (i_N+1).
\]
Hence, repetitions are allowed in $\cq^s$ and in every additive term there must be present every argument of $\cq^s$.

We notice that, by the above relations, the Riemann tensor of the hypersurface can be written as $\RR = \AAA * \AAA$, exploiting the above notation.

\section{Preliminary computations}

Let us recall the first variation formula for the functional $\cF_m$.

\begin{thm}[{\cite[Theorem~3.7]{Ma02}}]\label{thm:FirstVariation}
	Let $\vfi_t:M^n\to\R^{n+1}$ be a smooth family of immersions smoothly depending on $t\in(-\ep,\ep)$ and $X_t=\partial_t \vfi_t$. Then, for every $t\in(-\ep,\ep)$, there holds
	\[
	\frac{d}{dt} \cF_m(\vfi_t) = \int_M \EE_m(\vfi_t)  \scal{\nu, X_t} \,d\mu_t,
	\]
	with
	\[
	\EE_m(\vfi_t)=  2(-1)^{m} \De^m \HHH + \cq^{2m+1}(\nabla \nu, \AAA) +\HHH ,
	\]
	where all the quantities are relative to the hypersurface $\varphi_t$.
\end{thm}

The next lemma states the evolution formulae for the geometric quantities that we need in the computation of the second variation of the functional $\cF_m$.

\begin{lemma}\label{lem:Derivate}
Let $\vfi_t:M^n\to\R^{n+1}$ be a smooth family of immersions smoothly depending on $t\in(-\ep,\ep)$ and $\vfi=\vfi_0$. Let $X=\partial_t \vfi_t|_{t=0}$ and assume that $X$ is a {\em normal} vector field along $\vfi$. Then, we have
	\[
	\partial_t g_{ij}|_{t=0} = 2 \scal{\nu,X}  h_{ij} \,,
	\]
	\[
	\partial_t g^{ij}|_{t=0} = -2 \scal{\nu, X} g^{ik}g^{jl}  h_{kl} \,,
	\]
	\[
	\partial_t \nu|_{t=0} = -\nabla\scal{\nu, X} \,,
	\]
	\[
	\partial_t \Gamma^k_{ij}\bigr|_{t=0} = \nabla \AAA * \scal{\nu , X} + \AAA * \nabla \scal{\nu, X} \,,
	\]
\begin{equation}
\partial_th_{ij}|_{t=0}=-\nabla^2_{ij}\scal{\nu, X}+\scal{\nu, X} h^2_{ij}\,,\label{eq:DerivataB}
\end{equation}
\begin{equation}
\partial_t \HHH |_{t=0}=-\Delta\scal{\nu, X} - \scal{\nu, X} |\AAA|^2\,,\label{eq:DerivataH}
\end{equation}
\begin{equation}
\partial_t \De^m f |_{t=0} - \De^m \partial_t f  |_{t=0} = \cp_{2m} (f_0, \AAA, \scal{\nu, X})\,,\label{eq:DerivataLapl}
\end{equation}
for any smooth function $f \in C^\infty(M\times (-\ep,\ep))$ and $m\in \N$ with $m\ge 1$, where $f_0=f(\cdot,0)$ and
\begin{equation}
\partial_t \cq^{2m+1}(\nabla\nu, \AAA)|_{t=0} = \cq^{2m+3}(\scal{\nu,X}, \nabla\nu, \AAA) + \scal{\nu,X} \cq^{2m+2}(\nabla\nu, \AAA)\,.\label{eq:DerivataQ}
\end{equation}
\end{lemma}
\begin{proof}
The first four formulae are computed explicitly at page~150 of~\cite{Ma02}.

By means of the Gauss--Weingarten relations~\eqref{gwein}, setting $X=\beta\nu$, hence $\scal{\nu, X}=\beta$, we compute
\begin{align*}
\partial_th_{ij}|_{t=0}=&\,-\partial_t\langle\nu,\partial^2_{ij}\varphi_t\rangle|_{t=0}\\
=&\,-\langle\nu,\partial^2_{ij}(\beta \nu)\rangle+\langle\nabla\beta,\partial^2_{ij}\varphi\rangle\\
\,=&\,-\partial^2_{ij}\beta-\beta\langle\nu,\partial_i(h_{jl}g^{lk}\partial_k\varphi)\rangle
+\langle\partial_l\beta g^{ls}\partial_s\varphi,\Gamma_{ij}^k\partial_k\varphi-h_{ij}\nu\rangle\\
\,=&\,-\partial^2_{ij}\beta-\beta\langle\nu,h_{jl}g^{lk}\partial^2_{ik}\varphi\rangle
+\partial_k\beta\Gamma_{ij}^k\\
\,=&\,-\nabla^2_{ij}\beta+\beta h_{ik}g^{kl}h_{lj}
\end{align*}
that is, $\partial_th_{ij}|_{t=0}=-\nabla^2_{ij}\scal{\nu, X}+\scal{\nu, X} h^2_{ij}$, hence it follows
\begin{equation}
\partial_t \HHH |_{t=0}=\partial_t (g^{ij} h_{ij})|_{t=0} =-2 \scal{\nu, X} |\AAA|^2 - \Delta\scal{\nu, X} +\scal{\nu, X}|\AAA|^2=-\Delta\scal{\nu, X} - \scal{\nu, X} |\AAA|^2.
\end{equation}

We now deal with equation~\eqref{eq:DerivataLapl} arguing by induction on $m\ge 1$. Using the previous evolution formulae, for $m=1$ we compute
	\[
	\begin{split}
	\partial_t \De f |_{t=0}
	&= \partial_t (g^{ij}(\partial^2_{ij} f - \Gamma^k_{ij} \partial_k f ))  |_{t=0}\\
	&= - 2 \scal{\nu, X} g^{ik}g^{jl}  h_{kl} \nabla^2_{ij} f_0
	+ \De \partial_t|_{t=0} f - g^{ij} (\nabla \AAA * \scal{\nu , X} + \AAA * \nabla \scal{\nu, X}) \partial_k f_0,
	\end{split}
	\]
	and the claim follows. Now for $m+1\ge1$, by induction we get
	\[
	\begin{split}
	\partial_t\De^{m+1} f  |_{t=0}
	&= \De (\partial_t \De^{m} f)|_{t=0}  + \cp_2(\De^m f_0, \AAA , \scal{\nu, X} ) \\
	&= \De \left( \De^m \partial_t f|_{t=0} + \cp_{2m} (f_0, \AAA, \scal{\nu, X}) \right) + \cp_{2m+2}( f_0, \AAA , \scal{\nu, X} ).
	\end{split}
	\]
	
Finally, in order to show equation~\eqref{eq:DerivataQ}, we need to differentiate a generic term of the form
\[
\circledast_{k=1}^N \nabla^{i_k}\nabla\nu \circledast_{l=1}^M \nabla^{j_l} \AAA,
\]
with $\sum_{k=1}^N (i_k+1) + \sum_{l=1}^M (j_l+1) = 2m+1$.\\
For any component $\nu^\alpha$ of $\nu$ we can apply~\cite[Proposition~3.6]{Ma02} in order to get
\[
\partial_t (\nabla^{i_k}\nabla\nu^\alpha)|_{t=0} = - \nabla^{i_k+1} \nabla^\alpha \scal{\nu,X} + \cp_{i_k} (\scal{\nu,X}, \nabla \nu, \AAA),
\]
where $\nabla^\alpha \scal{\nu,X}$ denotes the $\alpha$--th component in $\R^{n+1}$ of the gradient $\nabla\scal{\nu, X}$.
Also, by~\cite[Lemma~3.5]{Ma02} and formula~\eqref{eq:DerivataB}, we have
\[
\begin{split}
\partial_t (\nabla^{j_l} \AAA)|_{t=0}
&= \nabla^{j_l} ( -\nabla^2 \scal{\nu, X} + \scal{\nu, X} \AAA *\AAA ) + \cp_{j_l} (\AAA, \AAA, \scal{\nu, X} ) \\
&= -\nabla^{j_l+2}\scal{\nu, X} + \cp_{j_l} (\AAA, \AAA, \scal{\nu, X} ).
\end{split}
\]
Therefore, using these formulae and the ones above for the derivative of the metric $g_{ij}$ and its inverse $g^{ij}$, formula~\eqref{eq:DerivataQ} follows.
\end{proof}

We can now compute the second variation of $\cF_m$.

\begin{thm}\label{thm:SecondVariation}
Let $\vfi_t:M^n\to\R^{n+1}$ be a smooth family of immersions smoothly depending on $t\in(-\ep,\ep)$. Denote $\vfi=\vfi_0$ and assume that $\vfi$ is a critical point for $\cF_m$, i.e., $\EE_m(\vfi)=0$. Let $X=\partial_t \vfi_t|_{t=0}$ and assume that $X$ is normal along $\vfi$. Then
\[
\frac{d^2}{dt^2} \cF_m (\vfi_t)\bigg|_{t=0} = \int_M \left( 2(-1)^{m+1} \De^{m+1} \scal{\nu, X} + \Omega(\scal{\nu,X}) \right) \scal{\nu,X} \, d\mu,
\]
where $\Omega(\scal{\nu,X})$ is linear in $\scal{\nu,X}$ and depends on its covariant derivatives of order $2m$ at most.
\end{thm}

\begin{proof}
By Theorem~\ref{thm:FirstVariation} we have
\[
\begin{split}
    \frac{d^2}{dt^2} \cF_m (\vfi_t)\bigg|_{t=0}
    &= \frac{d}{dt} \int_M \EE_m(\vfi_t)  \scal{\nu, \partial_t\vfi_t} \,d\mu_t \bigg|_{t=0} 
    = \int_M \left[\frac{\partial}{\partial t}  \EE_m(\vfi_t) \right]\bigg|_{t=0}  \scal{\nu, X} \,d\mu,
\end{split}
\]
as $\EE_m(\vfi)=0$. Using the explicit expression for $\EE_m(\vfi_t)$  (Theorem~\ref{thm:FirstVariation}), applying formula~\eqref{eq:DerivataLapl} with $f=\HHH$ and equations~\eqref{eq:DerivataH},~\eqref{eq:DerivataQ}, we get
\begin{align*}
    \frac{d}{dt}   \EE_m(\vfi_t)\bigg|_{t=0}
    &= 2(-1)^{m+1}\Delta^m (  \De\scal{\nu, X} + \scal{\nu, X} |\AAA|^2) + \cp_{2m}(\HHH,\AAA,\scal{\nu,X})\\
    &\indent  + \cq^{2m+3}(\scal{\nu,X}, \nabla\nu, \AAA)  + \scal{\nu,X} \cq^{2m+2}(\nabla\nu, \AAA) -  (  \De\scal{\nu, X} + \scal{\nu, X} |\AAA|^2) \\
    &= 2(-1)^{m+1} \De^{m+1}\scal{\nu, X} + 2(-1)^{m+1} \Delta^m(\scal{\nu,X}|A|^2)\\
    &\indent  + \cq^{2m+3}(\scal{\nu,X}, \nabla\nu, \AAA)  + \scal{\nu,X} \cq^{2m+2}(\nabla\nu, \AAA) -  (  \De\scal{\nu, X} + \scal{\nu, X} |\AAA|^2) .
\end{align*}
Hence, the thesis follows by observing that a generic monomial in $\cq^{2m+3}(\scal{\nu,X}, \nabla\nu, \AAA)$ is of the form
\[
\circledast_{k=1}^N \nabla^{i_k} \scal{\nu, X} 
\circledast_{l=1}^M \nabla^{j_l}  \nabla\nu
\circledast_{s=1}^P \nabla^{r_s}  \AAA,
\]
with
\[
\sum_{k=1}^N (i_k+1) +\sum_{l=1}^M (j_l+1) +\sum_{s=1}^P (r_s+1) = 2m+3,
\]
and $N,M,P\ge1$ and then $i_k\le 2m$ for any $k$.
\end{proof}

It follows that, by polarization, we can define the bilinear form
\begin{equation}\label{eq:SecondVariation}
    \begin{split}
    (\delta^2 \cF_{m})_{\vfi}(f_1,f_2) & \coloneqq \frac{d}{ds} \frac{d}{dt}\cF_m(\vfi + s f_1 \nu + t f_2 \nu ) \bigg|_{t=0} \bigg|_{s=0} \\ 
    & \:= \int_M \left( 2(-1)^{m+1} \De^{m+1} f_1+ \Omega(f_1) \right) f_2 \, d\mu\,,
\end{split}
\end{equation}
for any pair of smooth functions $f_1,f_2:M\to\R$ and $\Om$ is as in Theorem~\ref{thm:SecondVariation}.

\section{Analysis of the second variation}

Suppose that $\vfi:M\to\R^{n+1}$ is a smooth critical point of $\cF_m$, i.e., $\EE_m(\vfi)=0$. The formula for the second variation given above shows that $(\delta^2\cF_m)_\vfi (f_1,f_2)$ is well-defined for $f_1\in W^{2m+2,2}(M,g)$ and $f_2 \in L^2(\mu)$. This means that
\[
(\delta^2\cF_m)_\vfi (f,\cdot) \in L^2(\mu)^\star,
\]
for any $f \in W^{2m+2,2}(M,g)$ and it is well-defined the map
\[
W^{2m+2,2}(M,g) \ni f \mapsto (\delta^2\cF_m)_\vfi (f,\cdot) \in L^2(\mu)^\star.
\]
We are going to exploit the theory of Fredholm operators between Banach spaces. For definitions and results on the subject we refer the reader to~\cite[Section~19.1]{HormanderIII}. We recall that if $T:V_1\to V_2$ is a Fredholm operator between Banach spaces, its index is defined to be the integer number
\[
{\rm index } \,T\coloneqq \dim \ker T - \dim \, {\rm coker }\,T.
\]
where $\dim$ denotes the dimension of a finite dimensional vector space.

\begin{prop}\label{prop:Fredholm}
Let $\vfi:M\to\R^{n+1}$ be a smooth critical point of $\cF_m$, i.e., $\EE_m(\vfi)=0$. Then the second variation functional
\[
(\delta^2\cF_m)_\vfi : W^{2m+2,2}(M,g) \to L^2(\mu)^\star
\]
is a Fredholm operator of index zero.
\end{prop}

In order to prove Proposition~\ref{prop:Fredholm} we need the following commutation rule.

\begin{lemma}\label{lem:CommutationRule}
Let $\vfi:M^n\to\R^{n+1}$ be a smooth immersion and let $T$ be a tensor defined on $M$. Assume $M$ is endowed with the pull-back metric $g$ induced by $\vfi$. Then
\[
\nabla\De^l T - \De^l \nabla T =  \cp_{2l-1}(\AAA,\AAA,T) ,
\]
for any $l\in \N$ with $l\ge 1$.
\end{lemma}

\begin{proof}
As we need to prove a pointwise identity, we can take a local coordinate frame $E_1,...,E_n$ which is orthonormal at a given point $p$ (that is, $\scal{E_i,E_j}=\de_{ij}$) and $\nabla_i E_j = 0$ at $p$. In this way we can compute
\[
\begin{split}
    (\De \nabla T ) (E_k) 
    &= ( \nabla^2 (\nabla T)  (E_i,E_i) ) (E_k)\\
    &=( \nabla_i (\nabla_i \nabla T ) - \nabla_{\nabla_i E_i} \nabla T )(  E_k ) \\
    &= ( \nabla_i (\nabla_i \nabla T ) ) (E_k) \\
    &=  \nabla_i ( (\nabla_i \nabla T )  (E_k) ) - (\nabla_i \nabla T ) (\nabla_i E_k) \\
    & = \nabla_i (\nabla^2 T (E_i,E_k)) - \nabla^2 T (E_i, \nabla_i E_k ) \\
    & = \nabla_i (\nabla^2 T (E_i,E_k)).
\end{split}
\]
at the point $p$. On the other hand, using that for any tensor $S$ we have the commutation rule 
$$
(\nabla^2S)(E_j,E_l)= (\nabla^2S)(E_l,E_j) + \RR * S
$$
for any $j$ and $l$, we obtain
\[
\begin{split}
    (\nabla\De T)(E_k) 
    &= \nabla_k ( {\rm trace\,} \nabla^2 T )\\
    &= {\rm trace \,} \nabla_k \nabla^2 T \\
    &= (\nabla_k (\nabla^2 T ) )(E_i,E_i) \\
    &= (\nabla^3 T) (E_k,E_i,E_i) \\
    &= (\nabla^3 T) (E_i,E_k,E_i) + \RR * \nabla T \\
    &= (\nabla_i (\nabla^2T) ) (E_k,E_i) + \RR * \nabla T\\
    &= \nabla_i (\nabla^2T ( E_k,E_i) ) - (\nabla^2T)(\nabla_i E_k, E_i) - (\nabla^2T)( E_k, \nabla_i E_i)  + \RR * \nabla T\\
    &= \nabla_i ( \nabla^2T ( E_i,E_k) + \RR * T ) + \RR * \nabla T\\
    &= (\De \nabla T) (E_k) + \nabla (\RR*T) + \RR * \nabla T \\
    &= (\De \nabla T) (E_k) + \cp_{1}(\AAA,\AAA, T),
\end{split}
\]
where we used that $\RR= \AAA *\AAA$, by Gauss equations.
Hence, the thesis is proved for $l=1$. Letting now $l+1\ge 1$, by induction we obtain
\[
\begin{split}
    \nabla \De \De^l T 
    &= \De \nabla \De^l T + \cp_1 (\AAA, \AAA, \De^l T) 
    = \De ( \De^l \nabla T + \cp_{2l-1}(\AAA, \AAA, T) ) + \cp_{2l+1} (\AAA, \AAA, T),
\end{split}
\]
and the thesis follows.
\end{proof}

We are now ready to prove Proposition~\ref{prop:Fredholm}. A relevant property about Fredholm operators that we are going to use is the following. If  $T:V_1\to V_2$ is a Fredholm operator between Banach spaces and $K:V_1\to V_2$ is a compact operator, then $T+K$ is Fredholm and ${\rm index } (T+K)={\rm index } \, T$ (see~\cite[Corollary~19.1.8]{HormanderIII}).

\begin{proof}[Proof of Proposition~\ref{prop:Fredholm}]
For $f_1 \in W^{2m+2,2}(M,g)$ the functional $(\delta^2\cF_m)_\vfi (f_1,\cdot)$ is given by
\[
(\delta^2\cF_m)_\vfi (f_1,f_2) = \scal{ \cL(f_1) , f_2 }_{L^2(\mu)},
\]
where $\cL: W^{2m+2,2}(M,g) \to L^2(\mu)$ is
\[
\cL(f) = 2(-1)^{m+1}\De^{m+1} f + \Om (f),
\]
and $\Om$ is as in Theorem~\ref{thm:SecondVariation}, hence $\Om$ is a compact operator. Therefore 
$$
(\delta^2\cF_m)_\vfi :  W^{2m+2,2}(M,g) \to L^2(\mu)^\star
$$
is Fredholm of index zero if and only if the same holds for $\cL :  W^{2m+2,2}(M,g) \to L^2(\mu)$.\\
We then claim that the operator
\[
C\id + 2(-1)^{m+1}\De^{m+1}: W^{2m+2,2}(M,g) \to L^2(\mu)
\]
is invertible for $C>0$ sufficiently large, thus it is Fredholm of index zero. As the inclusion $\id:  W^{2m+2,2}(M,g) \to L^2(\mu) $ is compact, this eventually implies that $ 2(-1)^{m+1}\De^{m+1}: W^{2m+2,2}(M,g) \to L^2(\mu)$ is Fredholm of index zero.\\
The injectivity of the above operator immediately follows, suppose indeed that we have $Cf + 2(-1)^{m+1}\De^{m+1} f =0$, if $m=2k+1$, multiplying by $f$ and integrating, we get
\[
C \int_M f^2 \, d\mu = -2 \int_M f \De^{2(k+1)} f \, d\mu = - 2 \int_M (\De^{k+1} f)^2 \, d\mu,
\]
then $f=0$. If instead $m=2k$, multiplying by $f$ and integrating we get
\[
C \int_M f^2 \, d\mu = 2 \int_M f \De^{2k+1} f \, d\mu = -2 \int_M |\nabla\De^k f |^2 \, d\mu ,
\]
then $f=0$ as well.\\
About the surjectivity, given $h \in L^2(\mu)$ we aim at finding $f \in W^{2m+2,2}(M,g)$ such that $Cf + 2(-1)^{m+1}\De^{m+1} f =h$. We shall minimize the functional
\[
A_m: W^{m+1,2}(M,g)\to \R
\]
defined by
\[
A_m(f) \coloneqq 
\begin{cases}
\int_M \bigl[\frac{C}{2} f^2 + (\De^{k+1} f)^2 - fh\bigr] \, d\mu & \mbox{ if } m=2k+1,\\
\int_M \bigl[\frac{C}{2} f^2 + |\nabla\De^{k} f|^2 - fh\bigr] \, d\mu & \mbox{ if } m=2k.
\end{cases}
\]
We can prove that $A_m$ is coercive on $W^{m+1,2}(M,g)$, up to choosing $C>0$ sufficiently large (depending on $m$ and the geometry of $(M,g)$).\\
We first consider the case $m=2k+1$. Integrating by parts in the integral $\int_M (\De^{k+1}f)^2\, d\mu$, that is, using the divergence theorem and applying the commutation rule of Lemma~\ref{lem:CommutationRule} we get
\[
\begin{split}
    \int_M (\De^{k+1}f)^2\,d\mu
    = \int_M &- \scal{\nabla \De^k f , \nabla \De^{k+1} f }\,d\mu \\
    = \int_M &\bigl[- \scal{ \De^k  \nabla f , \De^{k+1} \nabla f } 
    + \nabla \De^k f * \cp_{2(k+1)-1}(\AAA,\AAA, f)\\
    &\,\,+ \nabla \De^{k+1} f * \cp_{2k-1}(\AAA,\AAA,f)\bigr]\,d\mu \\
    = \int_M &\bigl[- \scal{ \De^k  \nabla f , \De^{k+1} \nabla f }  + \cp_{4k+2} (\AAA,\AAA,f,f)\bigr]\,d\mu \\
    = \int_M &\bigl[(-1)^{k+1} \scal{ \nabla^{k+1} f, \De^{k+1} \nabla^{k+1} f }  + \cp_{4k+2} (\AAA,\AAA,f,f)\bigr]\,d\mu \\
    = \int_M &\bigl[|\nabla^{2k+2}f|^2  + \cp_{4k+2} (\AAA,\AAA,f,f)\bigr]\,d\mu \\
    = \int_M &\bigl[|\nabla^{m+1} f|^2 + \cp_{2m} (\AAA,\AAA,f,f)\bigr]\,d\mu.
\end{split}
\]
Moreover, by definition of $\cp_s$, we can apply the divergence theorem on the integral $\int_M \cp_{2m} (\AAA,\AAA,f,f)\,d\mu$ in the above expression so that in the polynomial there appear derivatives of $f$ of order $m$ at most.\\
We recall that for any covariant tensor $T$ there holds the general inequality (see~\cite[Chapter~3, Section~7.6]{Aubin})
\begin{equation}\label{eq:Interpolation}
\| \nabla^l T \|_{L^2(\mu)} \le C_{l,m} \| \nabla^{m+1} T \|_{L^2(\mu)}^{\frac{l}{m+1}} \| T \|_{L^2(\mu)}^{\frac{m+1-l}{m+1}} \le \ep \| \nabla^{m+1} T \|_{L^2(\mu)} + C_{l,m}(\ep) \| T \|_{L^2(\mu)},
\end{equation}
for any $l\le m$ and $\ep>0$. Therefore we can estimate
\[
\int_M |\cp_{2m} (\AAA,\AAA,f,f) | \, d\mu
\le C_m(\|\AAA\|_\infty^2)  \sum \int_M |\nabla^{l_1}f||\nabla^{l_2}f|\,d\mu,
\]
where $l_1,l_2\le m$ and then
\[
\int_M |\cp_{2m} (\AAA,\AAA,f,f) | \, d\mu
\le \ep C_m(\|\AAA\|_\infty^2) \| \nabla^{m+1} f \|^2_{L^2(\mu)} + C_m(\|\AAA\|_\infty^2, \ep ) \|f\|^2_{L^2(\mu)}.
\]
Therefore, taking $\ep C_m(\|\AAA\|_\infty^2)<1/2$ and $C=C(m,\|\AAA\|_\infty^2)$ sufficiently large, we estimate
\[
A_m(f) \ge \overline{C} \int_M\bigl[f^2 + |\nabla^{m+1}f|^2 -h^2\bigr] \, d\mu,
\]
that by inequality~\eqref{eq:Interpolation} implies that $A_m$ is coercive on $W^{m+1,2}(M,g)$. Analogously, one can prove the coercivity of $A_m$ also in the case $m=2k$.\\
It follows that there exists a function $F \in W^{m+1,2}(M,g)$ solving
\[
\int_M \bigl[CF f + 2 \De^{k+1} F \De^{k+1} f\bigr] \, d\mu = \int_M fh \, d\mu \qquad \forall\, f \in W^{m+1,2}(M,g)
\]
if $m=2k+1$, or
\[
\int_M \bigl[CF f + 2 \scal{\nabla \De^k F, \nabla\De^k f}\bigr] \, d\mu = \int_M fh \, d\mu \qquad \forall\, f \in W^{m+1,2}(M,g)
\]
if $m=2k$.
In any case, $F$ is a weak solution to an elliptic equation with constant coefficients and datum $h \in L^2(\mu)$ (in the sense of~\cite[Point~(d), Page~85]{Aubin}). Therefore, the standard regularity theory for distributional solutions applies (see~\cite[Theorem, Page~85]{Aubin}), hence $F$ belongs to $W^{2m+2,2}(M,g)$. Integrating by parts, we then get that $F$ solves $CF + 2(-1)^{m+1}\Delta^{m+1} F = h$, as required.
\end{proof}

\section{Convergence}

Suppose that $\vfi:M\to\R^{n+1}$ is a smooth critical point of $\cF_m$, that is, $\EE_m(\vfi)=0$. Then for $\rho_0>0$ suitably small, it is well-defined the functional $\cE_m:B_{\ro_0}(0)\subseteq W^{2m+2,2}(M,g)\to \R$ given by
\[
\cE_m(f) \coloneqq \cF_m(\vfi + f\nu).
\]
The advantage of the above definition is that the functional $\cE_m$ is now defined on an open set of a Banach space and we can then look at first and second variation functionals in the classical sense of functional analysis. More precisely, by Theorem~\ref{thm:FirstVariation} we have
\[
(\delta\cE_m)_{f_1}(f_2) \coloneqq \frac{d}{dt} \cE_m(f_1+tf_2) \Big|_{t=0}= \int_M \EE_m(\vfi+f_1\nu) \scal{\nu_1,  \nu} \,f_2 \, d\mu_1,
\]
where  $\nu$ (resp. $\nu_1$) is a unit normal vector along $\vfi$ (resp. $\vfi+f_1 \nu$) and $\mu_1$ is the volume measure induced by $\vfi+f_1\nu$. In this way we see that
\[
\delta \cE_m : B_{\ro_0}(0)\subseteq W^{2m+2,2}(M,g) \to L^2(\mu)^\star.
\]
Analogously, by Theorem~\ref{thm:SecondVariation} and formula~\eqref{eq:SecondVariation} the second variation of $\cE_m$ evaluated at $0 \in B_{\ro_0}(0)$ is given by
\[
(\delta^2\cE_m)_0 (f_1,f_2) 
= \int_M \left( 2(-1)^{m+1} \De^{m+1} f_1+ \Omega(f_1) \right) f_2 \, d\mu,
\]
for $\Omega$ as in Theorem~\ref{thm:SecondVariation}, so that
\[
(\delta^2\cE_m)_0 : W^{2m+2,2}(M,g) \to L^2(\mu)^\star,
\]
and it is a Fredholm operator of index zero by Proposition~\ref{prop:Fredholm}.

\medskip

In this setting we can apply the following abstract result stating sufficient conditions implying a \L ojasiewicz--Simon gradient inequality.

\begin{prop}[{\cite[Corollary~2.6]{Po20Loja}}]\label{prop:LojaAbstract}
Let $E:B_{\ro_0}(0)\subseteq V \to \R$ be an analytic map, where $V$ is a Banach space. Suppose that $0$ is a critical point for $E$, i.e., $\delta E_0 = 0$. Assume that there exists a Banach space $Z$ such that $V\hookrightarrow Z$, the first variation $\delta E : B_{\ro_0}(0)\to Z^\star$ is $Z^\star$--valued and analytic and the second variation $\delta^2 E_0 : V \to Z^\star$ evaluated at $0$ is $Z^\star$--valued and Fredholm of index zero.\\
Then there exist constants $C,\theta>0$ and $\alpha \in (0,1/2]$ such that
\[
|E(f)- E(0)|^{1-\alpha} \le C \| \delta E_f \|_{Z^\star},
\]
for every $f \in B_\theta(0) \subseteq V$.
\end{prop}

The above functional analytic result is a corollary of the useful theory developed in~\cite{Ch03} and it has been also proved in~\cite{Ru20} independently.

Applying Proposition~\ref{prop:LojaAbstract} to the functional $\cE_m$ we obtain the following corollary.

\begin{cor}\label{cor:Loja}
Let $\vfi:M\to\R^{n+1}$ be a smooth critical point of $\cF_m$, i.e., $\EE_m(\vfi)=0$. Let $\ro_0>0$ such that $\cE_m:B_{\ro_0}(0)\subseteq W^{2m+2,2}(M,g)\to \R$ is well-defined.\\
Then, there exist constants $C>0, \theta\in (0,\rho_0]$ and $\alpha \in (0,1/2]$ such that
\[
|\cF_m(\vfi + f\nu)- \cF_m(\vfi)|^{1-\alpha} \le C \| (\delta\cE_m)_f \|_{L^2(\mu)^\star},
\]
for every $f \in B_\theta(0) \subseteq W^{2m+2,2}(M,g)$.
\end{cor}

\begin{proof}
We want to apply Proposition~\ref{prop:LojaAbstract} with $V= W^{2m+2,2}(M,g)$ and $Z=L^2(\mu)$. By Proposition~\ref{prop:Fredholm} and the discussion at the beginning of the section, we just need to check that $\cE_m$ and $\delta\cE_m$ are analytic as maps between Banach spaces.

We can rewrite
\[
\cE_m(f) = \int_M 1 + \sum_{\alpha=1}^{n+1} \scal{ \nabla^m \nu^\alpha_f, \nabla^m\nu^\alpha_f }  \, d\mu_f
\]
where $\nu_f$ is a unit normal along $\vfi+f\nu$ and $\mu_f$ is the volume measure induced by $\vfi + f\nu$. If $\psi:M\to\R^{n+1}$ is any immersion, we have that a unit normal along $\psi$ is $\nu_\psi = \star \frac{\partial_1\psi \wedge \ldots \wedge \partial_n \psi}{|\partial_1\psi \wedge \ldots \wedge \partial_n \psi|}$, where $\star$ denotes the Euclidean Hodge star operator. As $\psi$ is an immersion, we see that $\psi \mapsto \nu_\psi$ is analytic. It follows that $f \mapsto \nu_f$ is analytic as well. As the metric tensor induced by an immersion $\psi:M\to\R^{n+1}$ has components $g_{ij}=\scal{\partial_i \psi, \partial_j \psi}$, we get that the metric tensor of $\vfi+f\nu$ depends analytically on $f$ and then it is analytic the dependence of $\mu_f$ and of Christoffel symbols (and thus of the connection) on $f$. Then the integrand in the definition of $\cE_m$ is just a sum of compositions and multiplications of functions which are analytic in $f$. Finally, integration is linear on $L^1(\mu)$, then $f \mapsto \cE_m(f) \in \R$ is analytic for $f \in B_{\rho_0} (0) \subseteq W^{2m+2,2}(M,g)$.

By the very same arguments, one can check that also $f\mapsto(\delta\cE_m)_f$ is analytic. Hence, all the hypotheses of Proposition~\ref{prop:LojaAbstract} are satisfied and the thesis follows.
\end{proof}

The starting point for proving the smooth convergence of the gradient flow of $\cF_m$ is the following sub--convergence theorem.

\begin{thm}[{\cite[Theorem~7.8, Theorem~8.2]{Ma02}}]\label{thm:Subconvergence}
Let $\vfi_0:M^n\to\R^{n+1}$ be a smooth immersion and let $m > \lfloor n/2 \rfloor$. Then there exists a unique smooth solution $\vfi:M\times [0,+\infty)\to \R^{n+1}$ to the evolution equation
\[
\begin{cases}
\partial_t \vfi = - \EE_m(\vfi_t) \nu_t,\\
\vfi(\cdot,0)=\vfi_0,
\end{cases}
\]
where $\nu_t$ denotes a unit normal vector field along $\vfi_t\coloneqq \vfi(\cdot,t)$.
Moreover, the solution satisfies the estimates
\begin{equation}\label{univest}
\|\nabla^k \AAA_t\|_{L^{\infty}(M,g_t)} \le C(k,n,\vfi_0),
\end{equation}
for any $t\in[0,+\infty)$, where $\AAA_t$ and $g_t$ are the second fundamental form and the metric of $\vfi_t$ respectively and there exists a smooth critical point $\vfi_\infty:M\to\R^{n+1}$ of $\cF_m$, a sequence of times $t_j\to+\infty$ and a sequence of points $p_j\in\R^{n+1}$ such that
\[
\| \vfi_{t_j} \circ \sigma_j - p_j - \vfi_\infty \|_{C^k(M)} \xrightarrow[j\to+\infty]{}0,
\]
for any $k \in \N$, where $\sigma_j$ is a sequence of diffeomorphisms of $M$.
\end{thm}

We need a preliminary lemma.

\begin{lemma}\label{lem:GraficoNormale}
Let $\vfi_0,\vfi_t,\vfi_\infty,\sigma_j,t_j,p_j$ be as in Theorem~\ref{thm:Subconvergence}.
Then, for any $\ep>0$ there is $j_\ep\in\N$ such that for any $j\ge j_\ep$ there exists $\delta_j>0$ such that the immersion $\vfi_t-p_j$ coincides with $\vfi_\infty + f_t\nu_\infty$ up to diffeomorphism, where $\nu_\infty$ is a unit normal vector along $\vfi_\infty$, for some ``height'' functions $f_t \in C^\infty(M)$ smoothly depending on $t\in[t_j,t_j+\delta_j)$. Moreover,
\[
\|f_t\|_{W^{2m+2,2}(M,g_\infty)} \le \ep,
\]
for any $t \in [t_j,t_j+\delta_j)$.
\end{lemma}
\begin{proof}
Fixed $\theta>0$ and $k>2m+2$, by Theorem~\ref{thm:Subconvergence} there is $j_\theta$ such that for any $j\ge j_\theta$ we have
\begin{equation}\label{eq:Jepsilon}
\| \vfi_{t} \circ \sigma_j - p_j - \vfi_\infty \|_{C^k(M)} <\theta,
\end{equation}
for every $t \in [t_j,t_j+\delta_j)$, for some $\delta_j>0$.\\
Let us assume that $\vfi_\infty$ is an embedding. The general statement analogously follows by recalling that immersions are local embeddings. So for $j_\theta$ large enough, $\vfi_t$ is an embedding as well for every $t \in [t_j,t_j+\delta_j)$. Moreover there exists $U\subseteq \R^{n+1}$ open set containing $N\coloneqq \vfi_\infty(M)$ such that it is well-defined the projection map $\pi:U\to N$ as
\[
\pi(p) = p - \frac12 \nabla^{\R^{n+1}} d_N^2(p),
\]
where $d_N$ is the distance function from $N$. The vector $\tfrac12 \nabla^{\R^{n+1}} d_N^2(p)$ is orthogonal to $N$ at $\pi(p)$, $\pi$ is smooth on $U$ and for $j_\theta$ sufficiently large we have that $(\vfi_t\circ \sigma _j(M)-p_j)\subseteq U$ for every $t \in [t_j,t_j+\delta_j)$ (for a proof of these facts see~\cite[Proposition~4.2]{MaMe03}).\\
Hence, for $x\in M$ the ``height'' function $f_t(x)$ is uniquely determined by the identity
\[
\vfi_t \circ \sigma_j (x) - p_j = \pi(\vfi_t\circ \sigma_j (x) - p_j ) + f_t(x) \nu_\infty (\vfi_\infty^{-1}\circ \pi \circ  (\vfi_t\circ \sigma_j (x) - p_j ) ),
\]
that is,
\begin{equation}\label{eq:DefFunzione}
f_t(x) = \scal{  \vfi_t \circ \sigma_j (x) - p_j - \pi(\vfi_t\circ \sigma_j (x) - p_j ) , \nu_\infty (\vfi_\infty^{-1}\circ \pi \circ  (\vfi_t\circ \sigma_j (x) - p_j ) )}.
\end{equation}
Then, the map $(x,t)\mapsto f_t(x)$ is smooth on $M\times [t_j,t_j+\delta_j)$ and $\|f_t\|_{W^{2m+2,2}(M,g_\infty)} \to 0$ as $\theta\to0$, by inequality~\eqref{eq:Jepsilon} and the fact that $k>2m+2$.\\
Hence, for the chosen $\ep>0$, taking a suitable $\theta>0$ we have the estimate in the statement of the lemma.
\end{proof}

We are now ready for proving our main result. The proof of Theorem~\ref{thm:Convergence} is essentially a generalization of the strategy employed in~\cite{MaPo20Loja} to show the smooth convergence of the elastic flow of closed curves in $\R^n$.

\begin{proof}[Proof of Theorem~\ref{thm:Convergence}.]
Let $\vfi_0,\vfi_t,\vfi_\infty,\sigma_j,t_j,p_j$ be as in Theorem~\ref{thm:Subconvergence}.
Fixed $k>2m+2$ and chosen $\ep>0$ smaller than the constant $\theta$ given by Corollary~\ref{cor:Loja}, relative to the critical point $\vfi_\infty$, by Theorem~\ref{thm:Subconvergence} and Lemma~\ref{lem:GraficoNormale}, there exists $j_\ep\in \N$ such that for every $j\ge j_\ep$ we have
\begin{equation}\label{eq:Jepsilon2}
\| \vfi_{t} \circ \sigma_j - p_j - \vfi_\infty \|_{C^k(M)} <\ep,
\end{equation}
for every $t \in [t_j,t_j+\delta_j)$ with some $\delta_j>0$, moreover, $\vfi_t\circ \sigma_j - p_j$ coincides with $\vfi_\infty + f_t\nu_\infty$, up to diffeomorphism, for the functions $f_t$ given by Lemma~\ref{lem:GraficoNormale} (we recall that $f_t \in C^{\infty}(M)$ depends on $j$), satisfying
\begin{equation}\label{eq:StimaSigma}
\|f_t \|_{W^{2m+2,2}(M,g_\infty)} < \ep<\theta,
\end{equation}
for every $t \in [t_j,t_j+\delta_j)$.\\
We claim that it is possible to choose $\ep>0$ small enough such that for any fixed $j\geq j_\ep$, the hypersurfaces $\vfi_t\circ \sigma_j - p_j$ coincide with $\vfi_\infty + f_t\nu_\infty$ (up to diffeomorphism) for some smooth functions $f_t$ with $\|f_t \|_{W^{2m+2,2}(M,g_\infty)} < \theta$ for any $t\in [t_j,+\infty)$.\\
We define
\[
H(t)\coloneqq |\cF_m(\vfi_t)-\cF_m(\vfi_\infty)|^{\alpha},
\]
where $\alpha\in (0,1/2]$ is as in Corollary~\ref{cor:Loja} applied to the critical point $\vfi_\infty$ and, without loss of generality, we can clearly assume that $H(t)>0$ for any $t$. As $\cF_m(\vfi_t) = \cF_m(\vfi_\infty+ f_t\nu_\infty)$, by Corollary~\ref{cor:Loja} we have
\[
\begin{split}
    H(t)^{\frac{1-\alpha}{\alpha}}
    &\le C \| (\delta\cE_m)_{f_t}\|_{L^2(\mu_\infty)^\star}\\ 
    &= C \Bigl( \int_M | \EE_m(\vfi_\infty + f_t \nu_\infty ) \scal{\nu_t,\nu_\infty} |^2 \det g_t \, d\mu_\infty \Bigr)^{1/2} \\
    &\le C(\vfi_\infty,\theta) \Bigl( \int_M | \EE_m(\vfi_\infty + f_t \nu_\infty )  |^2 \, d\mu_t \Bigr)^{1/2}  \\
    &= C(\vfi_\infty,\theta) \| \EE_m(\vfi_t) \|_{L^2(\mu)},
\end{split}
\]
where $\nu_t$, $g_t$, $\mu_t=\det g_t \, d\mu_\infty$ are the unit normal, metric tensor and volume measure on $\vfi_\infty + f_t\nu_\infty$ and we estimated $\sqrt{\det g_t}\le C(\vfi_\infty,\theta)$, for any $t\ge t_j$ such that 
$$
\|f_t \|_{W^{2m+2,2}(M,g_\infty)} < \theta.
$$
Differentiating $H$ and using the above inequality, we obtain
\[
\begin{split}
    \partial_t H(t) 
    &= \alpha H^{\frac{\alpha-1}{\alpha}} \partial_t \cF_m (\vfi_t) \\
    & = \alpha H^{\frac{\alpha-1}{\alpha}} \int_M \scal{ \EE_m(\vfi_t) , \partial_t^\perp \vfi_t } \, d\mu \\
    &= - \alpha H^{\frac{\alpha-1}{\alpha}} \int_M | \EE_m(\vfi_t) | | \partial_t^\perp \vfi_t | \, d\mu \\
    &= - \alpha H^{\frac{\alpha-1}{\alpha}} \| \EE_m(\vfi_t) \|_{L^2(\mu)} \| \partial_t^\perp \vfi_t\|_{L^2(\mu)} \\
    &\le - \alpha C(\vfi_\infty, \theta) \| \partial_t^\perp \vfi_t\|_{L^2(\mu)},
\end{split}
\]
for any $t\ge t_j$ such that $\|f_t \|_{W^{2m+2,2}(M,g_\infty)} < \theta$. For such times, possibly choosing a smaller $\ep$, we can assume that $|\nu_t- \nu_\infty|<1/2$. Letting $\widetilde{\vfi}_t\coloneqq \vfi_\infty + f_t \nu_\infty$ we thus get
\[
\begin{split}
    |\partial_t^\perp \widetilde{\vfi}_t| 
    &= |\scal{ \partial_t \widetilde{\vfi}_t, \nu_t} \nu_t|\\
    &= |\scal{ \partial_t \widetilde{\vfi}_t, \nu_{\infty}} \nu_t +\scal{ \partial_t \widetilde{\vfi}_t, \nu_t-\nu_\infty } \nu_t|\\
    & \ge |\scal{ \partial_t \widetilde{\vfi}_t, \nu_{\infty}}| - \frac12 | \partial_t \widetilde{\vfi}_t|\\
    &= \frac12 | \partial_t \widetilde{\vfi}_t|
\end{split}
\]
and the above estimate becomes
\[
    \partial_t H(t)\le - \alpha C(\vfi_\infty, \theta) \| \partial_t^\perp \vfi_t\|_{L^2(\mu)}= - \alpha C(\vfi_\infty, \theta) \| \partial_t^\perp \widetilde{\vfi}_t\|_{L^2(\mu_t)}\le - \alpha C(\vfi_\infty, \theta) \| \partial_t \widetilde{\vfi}_t\|_{L^2(\mu_t)}
\]
for any $t\ge t_j$ such that $\|f_t \|_{W^{2m+2,2}(M,g_\infty)} < \theta$. Integrating the above differential inequality and estimating $\sqrt{\det g_t} \ge C(\vfi_\infty,\theta)>0$, we obtain
\[
\begin{split}
    \| \widetilde{\vfi}_{\tau_2} - \widetilde{\vfi}_{\tau_1}  \|_{L^2(\mu_\infty)} 
    & = \left\| \int_{\tau_1}^{\tau_2} \partial_t \widetilde{\vfi}_t \, dt\, \,\right\|_{L^2(\mu_\infty)}\\
    &\le C(\vfi_\infty,\theta) \int_{\tau_1}^{\tau_2} \left\|\partial_t \widetilde{\vfi}_t \right\|_{L^2(\mu_t)} \, dt \\
    & \le C(\alpha,\vfi_\infty,\theta) ( H(\tau_1)- H(\tau_2) ) \\
    & \le C(\alpha,\vfi_\infty,\theta) ( \cF_m(\vfi_{\tau_1}) - \cF_m(\vfi_\infty) )^\alpha
\end{split}
\]
then, since possibly choosing a larger $j_\ep$ we can assume that $\cF_m(\vfi_{t_{j_\ep}}) - \cF_m(\vfi_\infty) \le \ep^{1/\alpha}$, we see that
\begin{equation}\label{eq:Cauchy}
    \| \widetilde{\vfi}_{\tau_2} - \widetilde{\vfi}_{\tau_1}  \|_{L^2(\mu_\infty)} \le C(\alpha,\vfi_\infty,\theta) \ep
\end{equation}
for any $\tau_2\ge \tau_1 \ge t_j$ such that $\|f_t \|_{W^{2m+2,2}(M,g_\infty)} < \theta$ on $t\in[t_j,\tau_2]$. Finally, since $\| f_t \|_{L^2(\mu_\infty)} = \| \widetilde{\vfi}_t - \vfi_\infty\|_{L^2(\mu_\infty)}$, we get
\begin{equation}\label{eq:CauchyFt}
	\| f_t \|_{L^2(\mu_\infty)} \le  \| \widetilde{\vfi}_t - \widetilde \vfi_{t_j}\|_{L^2(\mu_\infty)} +  \| \widetilde{\vfi}_{t_j} - \vfi_\infty\|_{L^2(\mu_\infty)} \le C(\alpha,\vfi_\infty,\theta) \ep
\end{equation}
for any $t\ge t_j$ such that $\|f_t \|_{W^{2m+2,2}(M,g_\infty)} < \theta$.\\
Since $m>\lfloor n/2 \rfloor$, estimate~\eqref{eq:StimaSigma} implies that the hypersurfaces $\widetilde{\varphi}_t$ are represented as graph on $\varphi_\infty$ by means of functions $f_t$ with uniformly equibounded gradients (such bound clearly depends on $\ep$ and goes to zero with it). Also, the inequalities~\eqref{univest} clearly hold also for the second fundamental form of the hypersurfaces $\varphi_t\circ\sigma_j$ and $\widetilde{\varphi}_t$, since they coincide with $\vfi_t$ up to diffeomorphism (and translation). These facts imply uniform estimates on the ``height'' functions $f_t$ in $W^{r,\infty}(M,g_\infty)$; namely, for any $r \in \N$ we have
\begin{equation}\label{eq:UniformEstimatesFt}
	\|f_t\|_{W^{r,\infty}(M,g_\infty)} \le C(r,n,\vfi_0,\vfi_\infty),
\end{equation}
for any $t \in [t_j,t_j+\delta_j)$ (a tedious but straightforward way to see this is to differentiate formula~\eqref{eq:DefFunzione} and use Gauss--Weingarten relations~\eqref{gwein}, taking into account that the closeness in $W^{2m+2,2}(M,g_\infty)$ implies that 
the metric tensor and the Christoffel symbols of the covariant derivative of $\widetilde{\varphi}_t$ are mutually ``comparable'' with the ones relative to $\varphi_\infty$). Hence, if $r>2m+2$ and $\ep>0$ is small enough, combining estimates~\eqref{eq:CauchyFt} and~\eqref{eq:UniformEstimatesFt}, the interpolation inequalities~\eqref{eq:Interpolation} imply that
\[
\|f_t \|_{W^{2m+2,2}(M,g_\infty)} < \theta,
\]
for any $t\in [t_j,t_j+\delta_j)$. By a maximality argument, it clearly follows that we can take $\delta_j=+\infty$, for every $j\geq j_\ep$. Hence, the estimate~\eqref{eq:Cauchy}, which then holds for any $t\geq t_j$, implies that the flow $\widetilde{\vfi}_t$ satisfies the Cauchy criterion for convergence in $L^2(\mu_\infty)$, hence $\widetilde{\vfi}_t$ converges in $L^2(\mu_\infty)$, as $t\to+\infty$.
Interpolating as before by means of inequalities~\eqref{eq:UniformEstimatesFt}, the same holds for $\widetilde{\vfi}_t$ in $W^{r,2}(M,g_\infty)$, for any $r\in \N$ and, by Sobolev embeddings, we thus deduce that there exists the limit $\lim_{t\to+\infty}\widetilde{\vfi}_t$ in $C^r(M)$ for any $r \in \N$. Therefore, the same conclusion holds for the original flow $\vfi_t$, up to diffeomorphism.
\end{proof}

\bibliographystyle{amsplain}
\bibliography{biblio}

\end{document}